\documentclass[11pt, a4paper]{article}

\usepackage{latexsym}
\usepackage{amsmath, epsfig}
\usepackage{color}
\usepackage{graphicx, float}
\usepackage{amssymb,amsfonts,amsthm,amsmath,graphicx,url}

\long\def\delete#1{}

\newtheorem{theorem}{Theorem}[section]
\newtheorem{lemma}[theorem]{Lemma}
\newtheorem{corollary}[theorem]{Corollary}
\newtheorem{definition}[theorem]{Definition}
\newtheorem{example}[theorem]{Example}

\input amssym.def
\input amssym.tex

\newcommand{\bmat}[1]{\begin{bmatrix}#1\end{bmatrix}}

\newcommand{\be}{\begin{equation}}
\newcommand{\ee}{\end{equation}}
\newcommand{\bea}{\begin{eqnarray}}
\newcommand{\eea}{\end{eqnarray}}
\newcommand{\bean}{\begin{eqnarray*}}
\newcommand{\eean}{\end{eqnarray*}}

\def\qed{\hfill$\Box$\vspace{11pt}}

\def\la{\langle}
\def\ra{\rangle}

\def\CCC{\mathbb{C}}

\def\ZZZ{\mathbb{Z}}

\def\QQQ{\mathbb{Q}}

\def\PP{{\cal P}}
\def\QQ{{\cal Q}}

\def\b0{{\bf 0}}

\def\bx{{\bf x}}

\def\Ga{\Gamma}

\def\Si{\Sigma}

\def\b{\beta}
\def\d{\delta}

\def\l{\lambda}

\def\vp{\varphi}

\def\Sym{{\rm Sym}}
\def\Inn{{\rm Inn}}

\def\Cay{{\rm Cay}}
\def\Aut{{\rm Aut}}

\def\SL{{\rm SL}}

\def\GF{{\rm GF}}

\def\det{{\rm det}}

\def\Ker{{\rm Ker}}

\title{Total perfect codes in Cayley graphs}
\author{Sanming Zhou \\
School of Mathematics and Statistics\\
The University of Melbourne\\
Parkville, VIC 3010, Australia\\
Email: smzhou@ms.unimelb.edu.au}

\begin{document}
\openup 0.4\jot
\maketitle

%\smallskip

\begin{abstract}
A total perfect code in a graph $\Ga$ is a subset $C$ of $V(\Ga)$ such that every vertex of $\Ga$ is adjacent to exactly one vertex in $C$. We give necessary and sufficient conditions for a conjugation-closed subset of a group to be a total perfect code in a Cayley graph of the group. As an application we show that a Cayley graph on an elementary abelian $2$-group admits a total perfect code if and only if its degree is a power of $2$. We also obtain necessary conditions for a Cayley graph of a group with connection set closed under conjugation to admit a total perfect code.  


\end{abstract}

\section{Introduction}

Let $\Ga$ be a graph with vertex set $V(\Ga)$ and edge set $E(\Ga)$, and let $e \ge 1$ be an integer. The \emph{ball} with centre $v \in V(\Ga)$ and radius $e$ is the set of vertices of $\Ga$ with distance at most $e$ to $v$ in $\Ga$. A \emph{code} in $\Ga$ is simply a subset of $V(\Ga)$. A code $C \subseteq V(\Ga)$ is called a \emph{perfect $e$-code} \cite{K86} in $\Ga$ if the balls with centres in $C$ and radius $e$ form a partition of $V(\Ga)$, that is, every vertex of $\Ga$ is at distance no more than $e$ to exactly one vertex of $C$. A code $C$ is said to be a \emph{total perfect code} \cite{GHT} in $\Ga$ if every vertex of $\Ga$ has exactly one neighbour in $C$. In graph theory, the ball around $v$ with radius $e$ is also called the $e$-neighbourhood of $v$ in $\Ga$, a perfect $1$-code in a graph is called an \emph{efficient dominating set} \cite{DS03, KP12} or \emph{independent perfect dominating set} \cite{L01}, and a total perfect code is called an \emph{efficient open dominating set} \cite{HHS}. Similar to perfect codes, total perfect codes in graphs are fascinating objects of study \cite{HHS}. Moreover, they have potential applications in some practical domains, such as placement of Input/Output devices in a supercomputing network so that each element to be processed is at distance at most one to exactly one Input/Output device \cite{AD}. It is known \cite{GSS} that deciding whether a graph has a total perfect code is NP-complete.  

The notions above were evolved from the work in \cite{Biggs1}, which in turn has a root in coding theory. In the classical setting, a $q$-ary \emph{code} is a subset $C \subseteq S^n$, where $S$ is a nonempty finite set (the alphabet) of size $q$ and $S^n$ the set of $n$-tuples (words of length $n$) from $S$. A code $C$ is a \emph{perfect $e$-code} if every word in $S^n$ is at distance no more than $e$ to exactly one codeword of $C$, where the (Hamming) distance between two words is the number of positions in which they differ. In the case when $S$ is a finite field $\GF(q)$, any subspace of the linear space $\GF(q)^n$ is called a \emph{linear code}. Such a linear code can be expressed as $\{\bx \in \GF(q)^n: \bx M^T = 0\}$, where $\bx$ is treated as a row vector and $M$ is a matrix over $\GF(q)$ called the \emph{parity check matrix} of the code. See \cite{Heden1, vL} for surveys on perfect codes and related definitions in the classical setting. 

In \cite{Biggs1}, Biggs showed that the proper setting for the perfect code problem is the class of distance-transitive graphs. In fact, the $q$-ary perfect $e$-codes of length $n$ in the classical setting are precisely the perfect $e$-codes in Hamming graph $H(n, q)$, which is distance-transitive and is defined to have $q$-ary words of length $n$ as vertices and edges joining pairs of words of Hamming distance one. Perfect codes in distance-transitive graphs were studied in, for example, \cite{B, Biggs1, HS, Heden, Smith, Smith1, SE}. Since the fundamental work of Delsarte \cite{Del}, a great amount of work on perfect codes in distance-regular graphs and association schemes in general has been produced. Beginning with \cite{K86}, perfect codes in general graphs have also attracted considerable attention in the community of graph theory; see \cite{K, LN, LS, M, S, Z} for example. 

Perfect codes in Cayley graphs are especially charming objects of study. In \cite{MBG07} sufficient conditions for Gaussian and Eisenstein-Jacobi graphs to contain perfect $e$-codes were given; such graphs are certain Cayley graphs on quotients of the rings of Gaussian and Eisenstein-Jacobi integers, respectively. These conditions were proved to be necessary in \cite{Z15} in the more general setting of cyclotomic graphs. In \cite{MBG09} a certain Cayley graph on the integer quaternions right-modulo a fixed nonzero element was introduced and perfect 1-codes in it were constructed. In \cite{T04} it was proved that there is no perfect 1-code in any Cayley graph on $\SL(2, 2^f)$, $f > 1$ with respect to any connection set closed under conjugation. In \cite{DS03} a methodology for constructing E-chains of Cayley graphs was given and was used to construct infinite families of E-chains of Cayley graphs on symmetric groups, where an \emph{E-chain} is a countable family of nested graphs each containing a perfect 1-code. In \cite{E87} perfect 1-codes in a Cayley graph with connection set closed under conjugation were studied, yielding necessary conditions in terms of the irreducible characters of the underlying group. In \cite{L01} it was proved that a subset $C$ of a group $G$ closed under conjugation (or a normal subset as is called in \cite{L01}) is a perfect 1-code in a Cayley graph on $G$ if and only if there exists a covering from the Cayley graph to a complete graph such that $C$ is a fibre of the corresponding covering projection. Perfect 1-codes in \emph{circulants} (that is, Cayley graphs on cyclic groups) were studied in \cite{OPR07, KM13}.  

Total perfect codes have also attracted considerable attention in recent years. In \cite{AD} `lattice-like' total perfect codes in the lattice $\ZZZ^n$ were constructed, and the authors of this paper conjectured that these enumerate all possibilities of such codes. In \cite{D} total perfect codes in the lattice $\ZZZ^2$ and in the grid graphs on tori were studied. In \cite{GHT} it was proved that the tensor product of any number of simple graphs has a total perfect code if and only if each factor has a total perfect code. In \cite{CHKS, KG} the grid graphs that have total perfect codes were characterized. In \cite{KPY}, lexicographic, strong, and disjunctive products of graphs admitting total perfect codes were characterized, and a similar result was also obtained for the cartesian product of any graph with the complete graph of two vertices. Total perfect codes are also related to diameter perfect codes, a notion introduced in \cite{AAK01} for distance regular graphs and adapted in \cite{Etzion11} for Lee metric over $\mathbb{Z}^n$ and $\mathbb{Z}_q^n$. In fact, when the Manhattan (for $\mathbb{Z}^n$) or Lee (for $\mathbb{Z}_q^n$) distance is considered, total perfect codes coincide with diameter perfect codes of minimum distance four, in the sense that if $C$ is a diameter perfect code then $C \cup gC$ is a total perfect code.  

Inspired by \cite{L01} and \cite{E87}, in this paper we prove a few results on total perfect codes in Cayley graphs. We first give necessary and sufficient conditions for a subset of a finite group closed under conjugation to be a total perfect code in a Cayley graph of the group (see Theorem \ref{thm:pseudo}), akin to \cite[Theorem 2]{L01} for perfect 1-codes. As a key component for this result and its proof, we introduce the concept of pseudocovers of graphs (see Definition \ref{def:pseudo}). As an application we show that a Cayley graph on an elementary abelian $2$-group admits a total perfect code if and only if its degree is a power of $2$ (see Theorem \ref{thm:cube-gen}). This extends \cite[Theorem 9.2.3]{G} from hypercubes to all Cayley graphs on elementary abelian $2$-groups, but our proof technique is different from that in \cite{G}. In \textsection\ref{sec:nec} we give two necessary conditions (see Theorems \ref{thm:nec-a} and \ref{thm:nec-b}) for a Cayley graph with connection set a union of conjugacy classes to contain a total perfect code, by using the irreducible characters of the underlying group. These are parallel to \cite[Theorems 6-7]{E87} and their proofs are accomplished by using a similar approach.

\section{Preliminaries}
\label{sec:pre}

All graphs in the paper are undirected without loops and multi-edges, and all groups considered are finite. Group-theoretic notation and terminology used can be found in most textbooks on group theory; see \cite{Rotman} for an introduction to group theory and \cite{Feit} for the theory of characters of finite groups. We use $1$ to denote the identity element of the group under consideration. An {\em involution} in a group is an element of order two. We use $\Ga[X]$ to denote the subgraph of a graph $\Ga$ induced by a subset $X$ of $V(\Ga)$, and $\Ga(v)$ the neighbourhood of a vertex $v \in V(\Ga)$ in $\Ga$ (that is, the set of vertices adjacent to $v$ in $\Ga$).

The following observation follows from the definition of a total perfect code. 

\begin{lemma}
\label{lem:trivial}
Let $\Ga$ be a graph. 
A subset $C$ of $V(\Ga)$ is a total perfect code in $\Ga$ if and only if 
\begin{itemize}
\item[\rm (a)] $\Ga[C]$ is a matching; and
\item[\rm (b)] $\{\Ga(v) \setminus C: v \in C\}$ is a partition of $V(\Ga) \setminus C$.
\end{itemize}
In particular, any total perfect code in $\Ga$ must contain an even number of vertices. 
\end{lemma}

A graph $\Si$ is called a {\em cover} of a graph $\Ga$ with \emph{covering projection} $p: \Si \rightarrow \Ga$ if there exists a surjective mapping $p: V(\Si) \rightarrow V(\Ga)$ such that for each $u \in V(\Si)$ the restriction of $p$ to $\Si(u)$ is a bijection from $\Si(u)$ to $\Ga(p(u))$. We call $\Si$ a {\em $k$-fold cover} of $\Ga$ if all \emph{fibres} $p^{-1}(v)$, $v \in V(\Ga)$ have size $k$. The following two lemmas are analogies of \cite[Lemmas 1-2]{L01}. 

\begin{lemma}
\label{lem:tri}
Let $\Ga$ be a $d$-regular graph, where $d \ge 1$.
\begin{itemize}
\item[\rm (a)] If $C$ is a total perfect code in $\Ga$, then 
\begin{equation}
\label{eq:div}
|C| = \frac{|V(\Ga)|}{d}.
\end{equation}
In particular, if $\Ga$ admits a total perfect code, then $d$ divides $|V(\Ga)|$ and $|V(\Ga)|/d$ is even. Thus any regular graph with an odd number of vertices does not admit total perfect codes. 
\item[\rm (b)] If $C_1, \ldots, C_n$ are pairwise disjoint total perfect codes in $\Ga$, then the subgraph of $\Ga$ with vertex set $\cup_{i=1}^n C_i$ and edges of $\Ga$ joining distinct such codes is a $c$-fold cover of $K_n$, where $c = |V(\Ga)|/d$.
\end{itemize}
\end{lemma}

\begin{proof}
(a) By Lemma \ref{lem:trivial}, we have $d |C| = |V(\Ga)|$, yielding $|C| = |V(\Ga)|/d$. Thus $d$ is a divisor of $|V(\Ga)|$ and again by Lemma \ref{lem:trivial}, $|V(\Ga)|/d$ must be even.

(b) By (a), all codes $C_1, \ldots, C_n$ have size $c  = |V(\Ga)|/d$. By the definition of a total perfect code, the edges of $\Ga$ between distinct $C_i$ and $C_j$ form a matching of size $c$. The union of these matchings for all pairs $(i, j)$, $i \ne j$, is precisely the subgraph $\Si$ of $\Ga$ with vertex set $\cup_{i=1}^n C_i$ and edges of $\Ga$ between distinct codes $C_i$. Let $K_n$ be the complete graph with vertices $v_1, \ldots, v_n$. Define $p: \cup_{i=1}^n C_i \rightarrow V(K_n)$ such that $p(u) = v_i$ if and only if $u \in C_i$. Then $p$ is a covering projection from $\Si$ to $K_n$ so that $\Si$ is a $c$-fold cover of $K_n$. 
\qed
\end{proof}
 
\begin{lemma}
\label{lem:covering}
Let $p: \Si \rightarrow \Ga$ be a covering projection. If $C$ is a total perfect code in $\Ga$, then $p^{-1}(C)$ is a total perfect code in $\Si$.
\end{lemma}

\begin{proof}
Since $C$ is a total perfect code in $\Ga$, by Lemma \ref{lem:trivial}, $\Ga[C]$ is a matching. Since $p$ is a covering projection, it follows that $\Si[p^{-1}(C)]$ is a matching. Similarly, since every vertex of $V(\Ga) \setminus C$ is adjacent to exactly one vertex of $C$, every vertex of $V(\Si) \setminus p^{-1}(C)$ is adjacent to at least one vertex of $p^{-1}(C)$. If a vertex $u \in V(\Si) \setminus p^{-1}(C)$ is adjacent to $v, v' \in p^{-1}(C)$ in $\Si$, then $p(u) \in V(\Ga) \setminus C$ is adjacent to $p(v), p(v') \in C$ in $\Ga$, which implies $p(v) = p(v')$ as  $C$ is a total perfect code in $\Ga$. Since $\Si$ is a cover of $\Ga$, we must have $v=v'$. Therefore, $p^{-1}(C)$ is a total perfect code in $\Si$.
\qed
\end{proof}

\begin{definition}
\label{def:pseudo}
{\em 
A graph $\Si$ is called a \emph{pseudocover} of a graph $\Ga$ if there exists a surjective mapping $p: V(\Si) \rightarrow V(\Ga)$ such that $\Si[p^{-1}(v)]$ is a matching for every $v \in V(\Ga)$ and $p$ is a covering projection from $\Si^*$ to $\Ga$, where $\Si^*$ is the graph obtained from $\Si$ by deleting the matching in each $\Si[p^{-1}(v)]$. We call $p$ a \emph{pseudocovering}, written $p: \Si \rightarrow \Ga$, and $p^{-1}(v)$, $v \in V(\Ga)$ the \emph{fibres} of $p$. 

A pseudocovering $p: \Si \rightarrow \Ga$ is called a \emph{$G$-pseudocovering} if $G$ is a subgroup of $\Aut(\Si)$ and there exists an isomorphism $h: \Ga \rightarrow \Si/\PP_G$ such that the quotient mapping $\Si \rightarrow \Si/\PP_G$ is the composition of $p$ and $h$, where $\PP_G$ is the partition of $V(\Si)$ into $G$-orbits and $\Si/\PP_G$ is the quotient graph of $\Si$ with respect to $\PP_G$.  
}
\end{definition}

The following lemma is the counterpart of \cite[Theorem 1]{L01} for total perfect codes. 

\begin{lemma}
\label{lem:pseudo}
Let $\Ga$ be a regular graph with $E(\Ga) \ne \emptyset$ and $n$ a positive integer. Then $\Ga$ is a pseudocover of $K_n$ if and only if $V(\Ga)$ admits a partition $\{C_1, \ldots, C_n\}$ such that each $C_i$ is a total perfect code in $\Ga$, $i=1, \ldots, n$.
\end{lemma}

\begin{proof}
The sufficiency follows from Lemma \ref{lem:tri}(b) and Definition \ref{def:pseudo} immediately. 

Suppose $\Ga$ is a pseudocover of $K_n$ with a pseudocovering projection $p: V(\Ga) \rightarrow V(K_n)$. Then for each $v \in V(K_n)$ the fibre $p^{-1}(v)$ is a total perfect code in $\Ga$, and all such fibres form a partition of $V(\Ga)$ as required.  
\qed
\end{proof}

\section{Total perfect codes in Cayley graphs}
\label{sec:main}

Given a group $G$ and a subset $S$ of $G$ such that $1 \not \in S$ and $S = S^{-1} := \{g^{-1}: g \in S\}$, the \textit{Cayley graph} $\Cay(G, S)$ on $G$ with respect to the \emph{connection set} $S$ is defined to have vertex set $G$ such that $x, y \in G$ are adjacent if and only if $xy^{-1} \in S$. Obviously, $\Cay(G, S)$ is an undirected graph of degree $|S|$ without loops. 

\begin{lemma}
\label{lem:partition}
Suppose $C \subseteq G$ is a total perfect code in a Cayley graph $\Cay(G, S)$. Then
\begin{itemize}
\item[\rm (a)] for every $g \in S$, $Cg$ is a total perfect code in $\Cay(G, S)$;
\item[\rm (b)] $\{gC: g \in S\}$ is a partition of $G$;
\item[\rm (c)] if $S$ is closed under conjugation, then $\{Cg: g \in S\}$ is a partition of $G$.
\end{itemize}
\end{lemma}

\begin{proof}
Denote $\Ga = \Cay(G, S)$.

(a) It is clear that any automorphism of $\Ga$ leaves the set of total perfect codes in $\Ga$ invariant. In particular, since for every $g \in S$, $\hat{g}: x \mapsto xg,\ x \in G$ defines an automorphism of $\Ga$, $Cg$ must be a total perfect code in $\Ga$.

(b) Suppose $g_1 C \cap g_2 C \ne \emptyset$ for distinct $g_1, g_2 \in S$. Then there exist $x_1, x_2 \in C$ such that $g_1 x_1 = g_2 x_2$. Since $g_1 \ne g_2$, we have $x_1 \ne x_2$. Thus $g_1 x_1$ is adjacent to distinct vertices $x_1, x_2$ in $C$, contradicting the assumption that $C$ is a total perfect code in $\Ga$. Thus $g_1C \cap g_2C = \emptyset$ for distinct $g_1, g_2 \in S$. Moreover, since $\Ga$ is $|S|$-regular, we have $|C||S| = |G|$ by (\ref{eq:div}) and therefore $\{gC: g \in S\}$ is a partition of $G$.

(c) Similar to (b), it suffices to prove $C g_1 \cap C g_2 = \emptyset$ for distinct $g_1, g_2 \in S$. Suppose otherwise, say,  $x_1 g_1 = x_2 g_2$ for $x_1, x_2 \in C$. Since $g_1 \ne g_2$, we have $x_1 \ne x_2$. Since $S$ is closed under conjugation, we have $h_1 := x_2 g_2 x_1^{-1} = x_1 g_1 x_1^{-1} \in S$ and $h_2 := x_1 g_1 x_2^{-1} = x_2 g_2 x_2^{-1} \in S$. Since $h_1 x_1 = x_2 g_2 = x_1 g_1 = h_2 x_2$, by (b) above we have $h_1 = h_2$ and hence $x_1 = x_2$, a contradiction. 
\qed
\end{proof}

The following result follows from Lemmas \ref{lem:pseudo} and \ref{lem:partition} immediately.

\begin{corollary}
\label{cor:pseudo}
Suppose $C \subseteq G$ is a total perfect code in a Cayley graph $\Cay(G, S)$. 
If $gC = Cg$ for every $g \in S$, then there exists a pseudocovering $p: \Cay(G, S) \rightarrow K_{|S|}$ such that $gC, g \in S$ are the fibres of $p$.  
Moreover, if $C$ is a normal subgroup of $G$, then $p$ is a $C$-pseudocovering.
\end{corollary}

The main results in this section are Theorem \ref{thm:pseudo} and Corollary \ref{cor:pseu} below, which are counterparts of \cite[Theorem 2]{L01} and \cite[Corollary 2]{L01}, respectively. As usual we denote $S^2 := \{gg': g, g' \in S\}$ for a subset $S$ of a group. 

\begin{theorem}
\label{thm:pseudo}
Suppose $C$ is a subset of a group $G$ closed under conjugation. Then the following are equivalent:
\begin{itemize}
\item[\rm (a)] $C$ is a total perfect code in $\Cay(G, S)$; 
\item[\rm (b)] there exists a pseudocovering $p: \Cay(G, S) \rightarrow K_{|S|}$ such that $gC$ is a fibre of $p$ for at least one element $g \in S$;
\item[\rm (c)] $C$ satisfies
$$
|C| |S| = |G|,\;\, C \cap \left((S^2 \setminus \{1\})C\right) = \emptyset.
$$
\end{itemize}
\end{theorem}

\begin{proof}
Denote $\Ga = \Cay(G, S)$.

(a) $\Rightarrow$ (b) This follows from Corollary \ref{cor:pseudo} immediately. 

(b) $\Rightarrow$ (a) Since $K_{|S|}$ is connected, all fibres of $p$ have the same size. Since one of these fibres is assumed to be $gC$ for some $g \in S$, all fibres of $p$ should have size $|gC| = |C|$. Hence $|C| |S| = |G|$. By Lemma \ref{lem:pseudo}, $gC$ is a total perfect code in $\Ga$. Since $C$ is closed under conjugation, we have $(gC)g^{-1} = C$. Since $g^{-1} \in S$, by Lemma \ref{lem:partition}(a), $C$ is a total perfect code in $\Ga$.  

(a) $\Rightarrow$ (c)  Suppose that $C$ is a total perfect code in $\Ga$. Since (a) and (b) are equivalent as proved above, we have $|C| |S| = |G|$ by the argument in the previous paragraph. It remains to prove that $C \cap \left((S^2 \setminus \{1\})C\right) = \emptyset$. Suppose otherwise. Then there exist $x_1, x_2 \in C$ and $g_1, g_2 \in S$ with $g_1 g_2 \ne 1$ such that $x_1 = g_1 g_2 x_2$. Thus $x_1 \ne x_2$ and $g_1^{-1} x_1 = g_2 x_2$. Hence $g_2 x_2$ is adjacent to distinct vertices $x_1, x_2$ in $C$, contradicting the assumption that $C$ is a total perfect code in $\Ga$.  

(c) $\Rightarrow$ (a) The assumption $C \cap \left((S^2 \setminus \{1\})C\right) = \emptyset$ implies $g_1C \cap g_2C = \emptyset$ for distinct $g_1, g_2 \in S$. This together with the assumption $|C| |S| = |G|$ implies that $\{gC: g \in S\}$ is a partition of $G$. Thus the neighbourhood of $C$ in $\Ga$ is given by
$$
\cup_{x \in C}\ \Ga(x) = \cup_{x \in C}\ Sx = \cup_{g \in S}\ gC = G.
$$
In other words, every element of $G \setminus C$ is adjacent to at least one element of $C$ in $\Ga$. If an element $z \in G \setminus C$ is adjacent to distinct $x_1, x_2 \in C$, then there exist $g_1, g_2 \in S$ such that $z = g_1 x_1 = g_2 x_2$. Since $x_1 \ne x_2$, we have $g_1 \ne g_2$ and so $x_1 =  g_1^{-1} g_2 x_2 \in C \cap \left((S^2 \setminus \{1\})C\right)$, contradicting the assumption $C \cap \left((S^2 \setminus \{1\})C\right) = \emptyset$. Therefore, every element of $G \setminus C$ is adjacent to a unique element of $C$ in $\Ga$. 

Since $\{gC: g \in S\}$ is a partition of $G$, every element $x \in C$ is of the form $x = g_1 x_1$ for some $g_1 \in S$ and $x_1 \in C$, and so $x$ is adjacent to $x_1$ in $\Ga$. If $x$ is also adjacent to some $x_2 \in C$ with $x_2 \ne x_1$, then $x = g_2 x_2$ for some $g_2 \in S$. So $g_1 \ne g_2$ and $g_1 x_1 = g_2 x_2 \in g_1 C \cap g_2 C$, which is a contradiction. Therefore, every element of $C$ is adjacent to a unique element of $C$ in $\Ga$. This together with what we proved in the previous paragraph implies that $C$ is a total perfect code in $\Ga$. 
\qed
\end{proof}

\begin{corollary}
\label{cor:pseu}
Let $N$ be a normal subgroup of a group $G$. Then the following are equivalent:
\begin{itemize}
\item[\rm (a)] $N$ is a total perfect code in $\Cay(G, S)$;
\item[\rm (b)] $\Cay(G, S)$ is an $N$-pseudocovering graph of $K_{|S|}$;
\item[\rm (c)] $N$ satisfies
$$
|G:N| = |S|,\;\; N \cap S^2 = \{1\}.
$$
\end{itemize}
Moreover, if one of these conditions holds, then $|N \cap S| = 1$ and the unique element of $N \cap S$ must be an involution.  
\end{corollary}

\begin{proof}
By Corollary \ref{cor:pseudo} and Theorem \ref{thm:pseudo}, to prove the equivalence of (a), (b) and (c), it suffices to prove that $N \cap \left((S^2 \setminus \{1\})N\right) = \emptyset$ if and only if $N \cap S^2 = \{1\}$. 

Suppose $N \cap \left((S^2 \setminus \{1\})N\right) = \emptyset$. Since $N$ is a subgroup of $G$, we have $1 \in N$ and so $N \cap (S^2 \setminus \{1\}) = \emptyset$. In other words, $N \cap S^2 = \{1\}$. 

Now suppose $N \cap \left((S^2 \setminus \{1\})N\right) \ne \emptyset$. Then there exist $x_1, x_2 \in N$ and $g_1, g_2 \in S$ such that $x_1 = g_1 g_2 x_2$ and $g_1 g_2 \ne 1$. Since $N$ is a subgroup of $G$, we have $x_1 x_2^{-1} = g_1 g_2 \in N \cap S^2$ and hence $N \cap S^2 \ne \{1\}$.

Suppose that one of (a)-(c) holds so that $N$ is a total perfect code in $\Cay(G, S)$. Since $1 \in N$ and the set of neighbours of $1$ in $N$ is $N \cap S$, we have $|N \cap S| = 1$ by the definition of a total perfect code. Let $g$ be the unique element of $N \cap S$. Then $g^2 \in N \cap S^2$ and so $g^2 = 1$ by (c).  
\qed
\end{proof}

Corollary \ref{cor:pseu} imposes strong conditions for a normal subgroup $N$ of $G$ to be a total perfect code in $\Cay(G, S)$. In the same fashion, one can show that for a normal subgroup $N$ of $G$ and an element $g \in S \setminus N$, $N \cup gN$ is a total perfect code in $\Cay(G, S)$ if and only if $|G:N| = 2|S|$, $N \cap S^2 = \{1\}$ and $N \cap S = gN \cap S^2 = g^{-1}N \cap S^2 = \emptyset$. 

Using additive notation for abelian groups, we write $0, S+S, -S$ and $C+x$ in place of $1, S^2, S^{-1}$ and $Cx$, respectively. Denote $C - C = \{g - g': g, g' \in C\}$ for any subset $C$ of an abelian group.

\begin{corollary}
\label{cor:abelian}
Let $G$ be an abelian group. A Cayley graph $\Cay(G, S)$ on $G$ admits a total perfect code if and only if there exists a subset $C$ of $G$ such that $|C| |S| = |G|$ and $(C - C) \cap (S+S) = \{0\}$; under these conditions $C$ is a total perfect code in $\Cay(G, S)$. Moreover, a subgroup $C$ of $G$ is a total perfect code in $\Cay(G, S)$ if and only if $|C| |S| = |G|$ and $C \cap (S+S) = \{0\}$.
\end{corollary}

\begin{example}
\label{ex:circ}
{\em 
Let $\Cay(\ZZZ_n, S)$ be a circulant graph, where $S \subseteq \ZZZ_n \setminus \{[0]\}$ satisfies $-S = S$. By Corollary \ref{cor:abelian}, $\Cay(\ZZZ_n, S)$ admits a total perfect code that is a subgroup of $\ZZZ_n$ if and only if $|S| = m$ is a divisor of $n$ and for any $[g_1], [g_2] \in S$, $g_1 + g_2$ is not a multiple of $m$ unless it is a multiple of $n$. Moreover, in this case $C = \{[km]: k \in \ZZZ\}$ is a total perfect code in $\Cay(\ZZZ_n, S)$. 

As a concrete example, the circulant $\Cay(\ZZZ_{18}, S)$ where $S=\{[1], [9], [17]\}$ admits $C = \{[0], [3], [6], [9], [12], [15]\}$ as a total perfect code because $|S|$ divides $18$ and $[0]$ is the only common element of $C$ and $S+S = \{[0], [2], [8], [10], [16]\}$. 

Similarly, $\Cay(\ZZZ_{20}, S)$ with $S=\{[1], [2], [10], [18], [19]\}$ admits $C = \{[0], [5], [10], [15]\}$ as a total perfect code. 
\qed
}
\end{example}

\section{Total perfect codes in cubelike graphs}
\label{sec:cube-gen}

The $n$-dimensional hypercube $Q_n$ is the Cayley graph on the elementary abelian 2-group $\ZZZ_2^n$ with respect to the set of vectors with exactly one nonzero coordinate. In general, any Cayley graph on $\ZZZ_2^n$, $n \ge 1$ is said to be \emph{cubelike} (a notion introduced by L. Lov\'{a}sz according to \cite{P}). The following result is a generalization of \cite[Theorem 9.2.3]{G}, where the same necessary and sufficient condition was given in the special case of hypercubes.  

\begin{theorem}
\label{thm:cube-gen} 
A connected cubelike graph admits a total perfect code if and only if its degree is a power of $2$. Moreover, we give a construction of linear total perfect codes in any cubelike graph with degree a power of $2$.  
\end{theorem}

\begin{proof} 
We identify $\ZZZ_{2}^{n}$ with the additive group of the $n$-dimensional linear space $V(n, 2)$ over $\GF(2)$. The vectors of $V(n, 2)$ are treated as row vectors and the zero vector of $V(n, 2)$ is denoted by ${\bf 0}_n$.
Since the operation of $\ZZZ_{2}^{n}$ is addition of vectors, we use $C + {\bf x}$ and $S+S$ in place of $C{\bf x}$ and $S^2$ respectively.   

Let $\Ga = \Cay(\ZZZ_{2}^{n}, S)$ be a connected cubelike graph with degree $d$. 
Then $S = \{{\bf u}_1, \ldots, {\bf u}_d\} \subseteq V(n, 2) \setminus \{{\bf 0}_n\}$ for $d$ distinct vectors ${\bf u}_1, \ldots, {\bf u}_d \in V(n, 2) \setminus \{{\bf 0}_n\}$. Since $\Ga$ is connected, $S$ must be a generating set of $\ZZZ_{2}^{n}$. In other words, $S$ contains a basis of $V(n, 2)$, and hence $n \leq d$. 

Suppose $\Ga$ admits a total perfect code. By Lemma \ref{lem:tri}, $d$ is a divisor of the order $2^{n}$ of $\Ga$. Thus $d = 2^t$ for some $1 \le t < n \le 2^t$ and the necessity is proved. 

To prove the sufficiency we assume that the degree of $\Ga$ is of the form $d = 2^t$ for some integer $t$ with $1 \le t < n \le 2^t$. By Corollary \ref{cor:pseu}, it suffices to prove the existence of a subgroup $C$ of $\ZZZ_{2}^{n}$ with index $|\ZZZ_{2}^{n}:C| = 2^t$ such that $C \cap (S+S) = \{{\bf 0}_n\}$. We achieve this by constructing such a subgroup $C$ explicitly with the help of an appropriate matrix over $\GF(2)$. 

Since $1 \le t < n \le d = 2^t$, there exists a $d \times n$ matrix $Q$ of rank $n$ over $\GF(2)$ such that the rows of $QP$ give all vectors of $V(t, 2)$, where 
$$
P = \bmat{I_{t}\\ 0_{(n - t) \times t}}
$$ 
with $I_{t}$ the identity matrix and $0_{(n - t) \times t}$ the zero-matrix of corresponding dimensions. (In fact, we may add $n - t$ column vectors of dimension $d$ to the $d \times t$ matrix whose rows are the vectors of $V(t, 2)$ such that the resultant matrix $Q$ has rank $n$.) 
Since $S$ contains a basis of $V(n, 2)$, the matrix $U$ with rows ${\bf u}_1, \ldots, {\bf u}_d$ is a $d \times n$ matrix of rank $n$. Since $Q$ and $U$ have the same dimension and rank, there exists a non-singular $n \times n$ matrix $R$ over $\GF(2)$ such that $Q = UR$. The non-singularity of $R$ implies that $M = RP$ is an $n \times t$ matrix with rank $t$. Thus the null space 
$\{{\bf x} \in V(n, 2): {\bf x}M = {\bf 0}_{t}\}$ of $M$ is an $(n - t)$-dimensional subspace of $V(n, 2)$. Therefore, its additive group $C$ is a subgroup of $\ZZZ_{2}^{n}$ with $|\ZZZ_{2}^{n}:C| = 2^t$. 
On the other hand, since the rows of $UM = U(RP) = QP$ give all vectors of $V(t, 2)$, the vectors ${\bf u}_i M$, $i = 1, \ldots, d$, are pairwise distinct. In other words, $({\bf u}_i + {\bf u}_j)M \ne {\bf 0}_{t}$ for $i \ne j$, or equivalently $C \cap (S+S) = \{{\bf 0}_n\}$. Therefore, by Corollary \ref{cor:pseu}, $C$ is a total perfect code in $\Ga$. Obviously, $C$ is a linear code with the transpose of $M$ as its parity check matrix. 
\qed
\end{proof}

The proof above gives an explicit construction of a linear total perfect code $C$ in the cubelike graph $\Cay(\ZZZ_{2}^{n}, S)$. By Lemma \ref{lem:partition}, the cosets $C + {\bf u}_i$, $i = 1, \ldots, d$, are all total perfect codes in $\Cay(\ZZZ_{2}^{n}, S)$ and they form a partition of the whole space $V(n, 2)$. 

It was observed by an anonymous referee of this paper that not every total perfect code in a cubelike graph with degree a power of $2$ is of the form $C + {\bf u}_i$ above. For example, let $n = 2^t$, $t \ge 4$, and let $C$ be a non-linear perfect code in $Q_{n-1}$ containing ${\bf 0}_{n-1}$ (there are many such non-linear perfect codes as seen in \cite{Heden1}). Then $\{0,1\} \times C$ is a total perfect code in $Q_n$, but it is not a coset of a linear total perfect code.

Theorem \ref{thm:cube-gen} implies the following result. 

\begin{corollary}
\label{thm:cube}
(\cite[Theorem 9.2.3]{G}) The hypercube $Q_d$ admits a total perfect code if and only if $d = 2^t$ for some integer $t \ge 1$.  
\end{corollary}

In \cite[Section 9.2.1]{G} the sufficiency was proved by showing that the direct sum of $\{0, 1\}$ and the well known Hamming code $H_t$ (that is, adding $0$ or $1$ at the beginning of each codeword of $H_t$) is a total perfect code of size $2^{n-t}$. This code is exactly the code $C$ constructed in the proof of Theorem \ref{thm:cube-gen} in the special case $Q_d = \Cay(\ZZZ_2^d, S)$ with $d = 2^t$, where $S = \{{\bf e}_1, \ldots, {\bf e}_d\}$ is the standard basis of $V(d, 2)$. This is because in this case the columns of the parity check matrix of $C$ (that is, the rows of $M$) are all vectors of $V(t,2)$ whilst the columns of the parity check matrix of $H_t$ are all nonzero vectors of $V(t,2)$. Note that $C + {\bf e}_i$, $i = 1, \ldots, d$, are also total perfect codes in $Q_d$. A related conjecture \cite[9.4.1]{G} asserts that, if $d = 2^t$, $t \ge 3$, then every dominating set of $Q_d$ with minimum size is a total perfect code. 

\begin{example}
\label{ex:Q4}
{\em 
By Corollary \ref{thm:cube}, $Q_4$ admits total perfect codes. Choose
$$
M=\bmat{1&0\\0&1\\1&1\\0&0}.
$$
Clearly $M$ has rank 2 and its rows are pairwise distinct. The additive group of the null space of $M$ is
$C_0 = \{(0, 0, 0, 0), (1, 1, 1, 0), (0, 0, 0, 1), (1, 1, 1, 1)\}$. 
By Corollary \ref{thm:cube}, $C_0$ is a total perfect code in $Q_4$. Moreover, by Lemma \ref{lem:partition}, the following are all total perfect codes in $Q_4$: $C_0 + {\bf e}_1 = \{(1, 0, 0, 0), (0, 1, 1, 0), (1, 0, 0, 1), (0, 1, 1, 1)\}$; $C_0 + {\bf e}_2 = \{(0, 1, 0, 0), (1, 0, 1, 0), (0, 1, 0, 1), (1, 0, 1, 1)\}$; $C_0 + {\bf e}_3 = \{(0, 0, 1, 0), (1, 1, 0, 0), (0, 0, 1, 1), (1, 1, 0, 1)\}$. 
(Note that $C_0 + {\bf e}_4 = C_0$.) These four codes form a partition of $V(4, 2)$; see Figure \ref{fig:Q4a} for an illustration. 
\qed
}
\end{example}

\vspace{-0.5cm}
\begin{figure}[ht]
\centering
\includegraphics*[height=7cm]{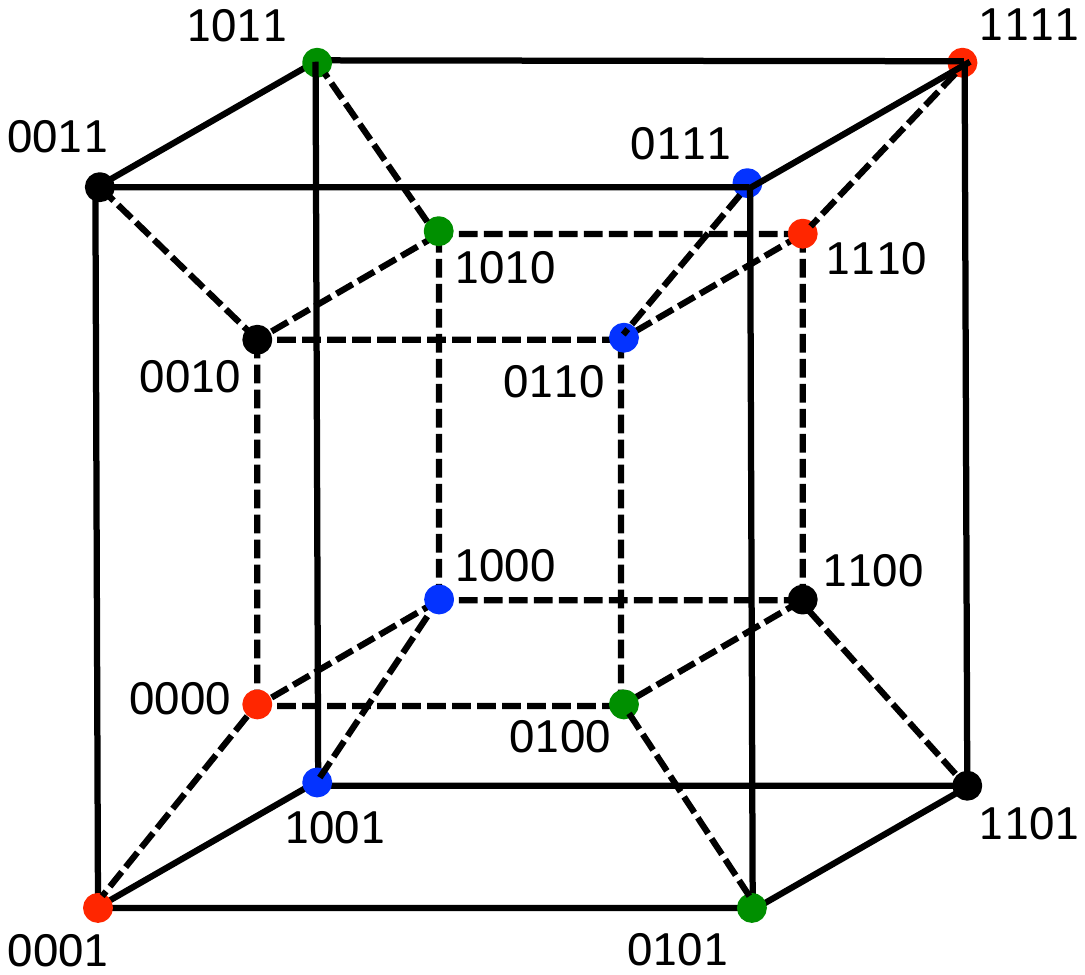}
\vspace{-0.5cm}
\caption{\small Total perfect codes in $Q_4$.} 
\label{fig:Q4a}
\end{figure} 

In contrast to $\ZZZ_2^n$, by Lemma \ref{lem:tri}(a), for any odd prime $p$, any Cayley graph on $\ZZZ_p^n$ does not admit any total perfect code.

\section{Necessary conditions}
\label{sec:nec}

Interesting necessary conditions for the existence of a perfect 1-code in a directed Cayley graph were given in \cite{E87} by using characters of the underlying group. In this section we prove analogous results for total perfect codes in undirected Cayley graphs with connection sets a union of conjugacy classes of the underlying group, by adapting the methodologies in \cite{E87}. We begin with a necessary condition for the existence of a total perfect code in a regular (but not necessarily Cayley) graph in terms of equitable partitions, and then apply it to two specific equitable partitions of the above-mentioned Cayley graphs.

\subsection{Total perfect codes and equitable partitions}
\label{subsec:eq}

In this subsection $\Ga = (V, E)$ is a $d$-regular graph, where $d \ge 1$, with adjacent matrix 
$$
A = (a_{uv})_{u, v \in V}.
$$

An \emph{equitable partition} \cite[Section 9.3]{GR} of $\Ga$ is a partition $\pi = \{V_1, \ldots, V_m\}$ of $V$ such that for every pair $i, j$, the number of neighbours of $u \in V_i$ in $V_j$ is a constant $b_{ij}$ independent of the choice of $u$. In \cite{E87} an equitable partition is called a regular partition. 

Denote by $\CCC V$ the vector space of functions from $V$ to the field of complex numbers $\CCC$, with addition and scalar multiplication defined in the usual way. As $V$ is finite, we may identify every $f \in \CCC V$ with the column vector $(f(v))_{v \in V}$. For each $v \in V$, define $\d_v \in \CCC V$ such that $\d_{v}(u) = 1$ if $u = v$ and $\d_{v}(u) = 0$ if $u \ne v$. Then every $f \in \CCC V$ can be written as $f = \sum_{v \in V} f(v) \d_v$. Define the `adjacency' linear mapping \cite{E87} by
\be
\label{eq:vphi}
\varphi: \CCC V \rightarrow \CCC V,\quad \d_v \mapsto \vp(\d_v) := \sum_{u \in \Ga(v)} \d_{u} = \sum_{u \in V} a_{uv} \d_{u}.
\ee
It is straightforward to verify \cite{E87} that the matrix of $\vp$ with respect to the `standard' basis $\{\d_v: v \in V\}$ of $\CCC V$ is exactly the adjacency matrix $A$ of $\Ga$. Alternatively, if $\CCC V$ is viewed as a vector space of column vectors, then $\vp$ maps $f \in \CCC V$ to $Af$. 

Define $\d_U = \sum_{v \in U} \d_v \in \CCC V$ for every $U \subseteq V$. That is, $\d_{U}(u) = 1$ if $u \in U$ and $0$ otherwise. In particular, $\d_V$ is the constant function $\mathbf{1}$ on $V$ (all-one column vector). It is not difficult to verify the following result. 

\begin{lemma}
\label{lem:equit}
(\cite[Lemma 2(ii)]{E87}; see also \cite[Lemma 9.3.2]{GR}) 
Let $\Ga = (V, E)$ be a $d$-regular graph. 
A partition $\{V_1, \ldots, V_m\}$ of $V$ is equitable if and only if the subspace of $\CCC V$ spanned by $\d_{V_1}, \ldots, \d_{V_m}$ is invariant under the linear mapping $\vp$. 
\end{lemma}

In the sequel, for an equitable partition $\pi = \{V_1, \ldots, V_m\}$ we denote by $\CCC V_{\pi}$ the subspace of $\CCC V$ spanned by $\d_{V_1}, \ldots, \d_{V_m}$. Then $\dim(\CCC V_{\pi}) = m$. Denote by $\vp_{\pi}$ the restriction of $\vp$ to $\CCC V_{\pi}$. It is straightforward to verify that the matrix $A_{\pi}$ of $\vp_{\pi}$ with respect to the `standard' basis $\d_{V_1}, \ldots, \d_{V_m}$ of $\CCC V_{\pi}$ is given by 
$$
A_{\pi} = (b_{ij})_{1 \le i,j \le m},
$$ 
where as before $b_{ij} = |\Ga(u) \cap V_j|$ for $u \in V_i$.  

\begin{lemma}
\label{lem:tp}
Let $\Ga = (V, E)$ be a $d$-regular graph. A subset $C \subseteq V$ is a total perfect code in $\Ga$ if and only if  
\be
\label{eq:ns}
\vp(\d_C) = \d_V.
\ee 
\end{lemma}

\begin{proof}
We have $\vp(\d_C)(w) = \sum_{v \in C} \vp(\d_v)(w) = \sum_{v \in C}a_{wv} = |\Ga(w) \cap C|$ for $w \in V$. From this and the definition of a total perfect code the result follows immediately. 
\qed
\end{proof}

We now prove the following counterpart of \cite[Proposition 3]{E87} by using a similar approach. 

\begin{theorem}
\label{thm:equ}
Let $\Ga = (V, E)$ be a $d$-regular graph and $C$ a total perfect code in $\Ga$. Then for every equitable partition $\pi = \{V_1, \ldots, V_m\}$ of $\Ga$ there exists a vector $(k_1, \ldots, k_m)^T \in \QQQ^m$ such that 
\be
\label{eq:eigen}
|V_i \cap C| = \left(\frac{1}{d} + k_i\right) |V_i|,\;\, i = 1, \ldots, m
\ee  
\be
\label{eq:eigen1}
A_{\pi} (k_1, \ldots, k_m)^T = (0, \ldots, 0)^T.
\ee
\end{theorem}
 
\begin{proof}
Define
$$
p: \CCC V \rightarrow \CCC V_{\pi}, \quad f \mapsto p(f) := \sum_{i=1}^{m} \left(\frac{\sum_{v \in V_i}f(v)}{|V_i|}\right) \d_{V_i}. 
$$
Then $p$ is a linear mapping from $\CCC V$ to $\CCC V_{\pi}$. 
Denote by $P$ the matrix of $p$ with respect to the `standard' basis $\{\d_v: v \in V\}$. It is straightforward to verify that the $(u, v)$-entry of $P$ is given by $P_{uv} = \d_{V_u}(v)/|V_u|$, where $V_u$ is the unique part $V_i$ of $\pi$ containing $u$. Using this and the assumption that $\pi$ is equitable, one can verify (see \cite{E87}) that $PA = AP$. In other words, $p \circ \vp = \vp \circ p$.    

Since $C$ is a total perfect code, we have $\vp(\d_C) = \d_V$ by (\ref{eq:ns}). This together with $p \circ \vp = \vp \circ p$ implies $(\vp \circ p) (\d_C) = (p \circ \vp)(\d_C) = p(\d_V) = \d_V$. Since $\Ga$ is $d$-regular, we have $\vp\left(\frac{1}{d}\d_V\right) = \d_V$ and hence $(\vp \circ p)\left(\frac{1}{d}\d_V\right) = (p \circ \vp)\left(\frac{1}{d}\d_V\right) = p(\d_V) = \d_V$. By the definition of $p$, we have $p(\d_C) - p\left(\frac{1}{d}\d_V\right) \in \CCC V_{\pi}$. Therefore, $\vp_{\pi} \left(p(\d_C) - p\left(\frac{1}{d}\d_V\right)\right) = \mathbf{0}$, that is, $p(\d_C) - p\left(\frac{1}{d}\d_V\right) \in \Ker(\vp_{\pi})$. On the other hand, since $p(\d_C) - p\left(\frac{1}{d}\d_V\right) \in \CCC V_{\pi}$, there exists $(k_1, \ldots, k_m)^T \in \CCC^m$ such that $p(\d_C) - p\left(\frac{1}{d}\d_V\right) = \sum_{i=1}^m k_i \d_{V_i}$. Since this vector is in $\Ker(\vp_{\pi})$ and $\vp_{\pi}$ has matrix $A_{\pi} = (b_{ij})_{1 \le i,j \le m}$ with respect to the basis $\d_{V_1}, \ldots, \d_{V_m}$, it follows that $(k_1, \ldots, k_m)^T$ satisfies (\ref{eq:eigen1}). We have $p(\d_C) = \sum_{i=1}^{m} \left(\frac{|V_i \cap C|}{|V_i|}\right) \d_{V_i}$ and $p\left(\frac{1}{d}\d_V\right) = \frac{1}{d}\d_V = \frac{1}{d} \sum_{i=1}^m \d_{V_i}$. Thus $\sum_{i=1}^{m} \left(\frac{|V_i \cap C|}{|V_i|}\right) \d_{V_i} - \frac{1}{d} \sum_{i=1}^m \d_{V_i} = \sum_{i=1}^m k_i \d_{V_i}$. Since $\d_{V_1}, \ldots, \d_{V_m}$ are independent, we have $\frac{|V_i \cap C|}{|V_i|} - \frac{1}{d} = k_i$ for each $i$. It is obvious that all coordinates $k_i$ are rationals. 
\qed
\end{proof}

\begin{corollary}
\label{cor:equ}
Let $\Ga = (V, E)$ be a $d$-regular graph. 
\begin{itemize}
\item[\rm (a)] If $\Ga$ admits a total perfect code, then for any equitable partition $\pi = \{V_1, \ldots, V_m\}$ of  $\Ga$, either $0$ is an eigenvalue of $A_{\pi}$ with an eigenvector $(k_1, \ldots, k_m)$ giving by (\ref{eq:eigen}), or $d$ divides $|V_i|$ for each $i = 1, \ldots, m$.  
\item[\rm (b)] If $\Ga$ admits a total perfect code and $d \ge 2$, then $0$ is an eigenvalue of $\Ga$ (obtained from the trivial equitable partition $\{\{v\}: v \in V\}$) and $(k_v)_{v \in V}$ is a corresponding eigenvector, where $k_v = 1-\frac{1}{d}$ if $v \in C$ and $k_v = -\frac{1}{d}$ if $v \notin C$. 
\end{itemize}
\end{corollary}

\begin{proof}
The truth of (a) follows from Theorem \ref{thm:equ} immediately. Applying (a) to the trivial equitable partition $\{\{v\}: v \in V\}$, we obtain (b) by noting that $d \ge 2$ is not a divisor of $1$. 
\qed
\end{proof}

Part (b) in Corollary \ref{cor:equ} is parallel to the well known result \cite[Lemma 9.3.4]{GR} that a regular graph admitting a perfect 1-code should have $-1$ as an eigenvalue.

\subsection{A necessary condition}
\label{subsec:a}

The purpose of this subsection is to establish the following analogy of \cite[Theorem 6]{E87} by using a similar approach. Denote by $d_{\chi}$ the degree of a character $\chi$ of a group. As $\chi$ is a class function, for a conjugacy class $K$ we write $\chi(K) =\chi(x)$ where $x \in K$. 

\begin{theorem}
\label{thm:nec-a}
Let $G$ be a group and $S$ a union of $s$ conjugacy classes of $G$ with $1 \not \in S$ and $S^{-1} = S$. Let $\QQ$ be the set of irreducible characters $\chi$ of $G$ such that $\sum_{K} \frac{|K|\chi(K)}{d_{\chi}} = 0$, where $K$ runs over all conjugacy classes of $G$ contained in $S$. If $\Cay(G, S)$ admits a total perfect code, then
\begin{itemize}
\item[\rm (a)] $\sum_{\chi \in \QQ} d_{\chi}^2 \ge |S| - 1$;
\item[\rm (b)] $|\QQ| \ge s$.
\end{itemize}
\end{theorem}

In the case when $G$ is abelian, it has exactly $|G|$ irreducible characters all of which are linear. Thus all $d_{\chi} = 1$ and $\QQ$ consists of those $\chi$ such that $\chi(S) := \sum_{g \in S} \chi(g) = 0$. Since $G$ is abelian, we have $|S| = s$ and $\chi(S)$, with $\chi$ running over all irreducible characters of $G$, are precisely the eigenvalues of $\Cay(G, S)$. Thus (a) and (b) above yield $|\QQ| \ge s-1$ and $|\QQ| \ge s$, respectively. Therefore, Theorem \ref{thm:nec-a} implies:

\begin{corollary}
\label{cor:nec-a}
Let $G$ be an abelian group and $S$ a subset of $G$ with $1 \not \in S$ and $S^{-1} = S$. Then $\Cay(G, S)$ admits a total perfect code only if the multiplicity of $0$ as an eigenvalue of $\Cay(G, S)$ is at least $|S|$ (that is, $G$ has at least $|S|$ irreducible characters $\chi$ such that $\sum_{g \in S} \chi(g) = 0$).
\end{corollary}

The rest of this subsection is devoted to the proof of Theorem \ref{thm:nec-a}. We assume that $G$ and $S$ are as in Theorem \ref{thm:nec-a} and denote $\Ga = \Cay(G, S)$. To be explicit we assume 
$$
S = \cup_{i=1}^s K_i,
$$ 
where each $K_i$ is a conjugacy class of $G$ with $K^{-1}_{i}$ also contained in $S$. Without loss of generality we may assume that the set of all conjugacy classes of $G$ is
$$
\pi = \{K_1, \ldots, K_s, K_{s+1}, \ldots, K_m\}.
$$
Then $\CCC G_{\pi}$ is the vector space of class functions of $G$. Since $S$ is closed under conjugation, the inner automorphism group $\Inn(G)$ of $G$ is a subgroup of $\Aut(\Ga)$ with respect to its natural action on $G$. Since the orbits of $\Inn(G)$ on $G$ are precisely the conjugacy classes of $G$, it follows that $\pi$ is an equitable partition of $\Ga$. We will apply Theorem \ref{thm:equ} to this particular partition in the proof of Theorem \ref{thm:nec-a}.  

Let 
$$
\l_G: G \rightarrow \Sym(\CCC G),\quad g \mapsto \l_G(g)
$$
be the left regular permutation representation of $G$, defined by
$$
(\l_G(g)f)(x) = f(g^{-1}x),\;\, f \in \CCC G,\;\, x \in G.
$$
It can be verified that
$$
\l_G(g)\d_v = \d_{gv},\;\,g, v \in G.
$$
Let $\vp: \CCC G \rightarrow \CCC G$ be the adjacency linear mapping as in (\ref{eq:vphi}) for $V=G$ and $\Ga = \Cay(G, S)$. 

\begin{lemma}
\label{lem:vphi}
$\varphi = \sum_{g \in S} \l_G(g)$.
\end{lemma}

\begin{proof}
We have $(\sum_{g \in S} \l_G(g))(\d_v) = \sum_{g \in S} \l_G(g)(\d_v) = \sum_{g \in S} \d_{gv} = \sum_{u \in \Ga(v)} \d_u = \varphi(\d_v)$ for any $v \in G$. Since $\{\d_v: v \in G\}$ is a basis of $\CCC G$, the result follows. 
\qed
\end{proof}

We remark that Lemma \ref{lem:vphi} holds for any $S \subset G$ that is not necessarily a union of conjugacy classes of $G$. 

\begin{lemma}
\label{lem:ev}
(\cite[Lemma 5]{E87}) 
The irreducible characters of $G$ (i) constitute a basis of $\CCC G_{\pi}$, and (ii) are eigenvectors of $\varphi_{\pi}$. More explicitly, for any irreducible character $\chi$ of $G$, 
\be
\label{eq:ir-eigen}
\varphi_{\pi}(\chi) = \left(\sum_{i=1}^s \frac{|K_i| \chi(K_i)}{d_{\chi}}\right) \chi.
\ee
\end{lemma}

The truth of (i) is a well known result in group theory as $\pi$ is the partition of $G$ into conjugacy classes. Part (ii) was proved in \cite{E87} with the help of \cite[(5.4)]{Feit}. 

\medskip
\begin{proof}\textbf{of Theorem \ref{thm:nec-a}}~
Suppose that $C$ is a total perfect code in $\Ga$. By Lemma \ref{lem:partition}(a), we may assume $1 \in C$ without loss of generality. Since $S$ is closed under conjugation, by Lemma \ref{lem:partition}(c), $\{Cg: g \in S\}$ is a partition of $G$, and hence $\d_{Cg}, g \in S$ are independent vectors of $\CCC G$. On the other hand, by Lemma \ref{lem:partition}(a) each $Cg$ with $g \in S$ is a total perfect code of $\Ga$. Thus by Lemma \ref{lem:tp}, $f = \d_{Cg}, g \in S$ are solutions to the linear equation $\vp(f) = \d_G$. Since these are $|S|$ independent solutions, the homogeneous equation $\vp(f) = 0$ has at least $|S|-1$ independent solutions. In other words, 
\be
\label{eq:ker}
\dim(\Ker(\vp)) \ge |S| - 1.
\ee

We now compute $\dim(\Ker(\vp))$ by way of the irreducible characters of $G$. Let $\chi_1, \ldots, \chi_m$ be such characters (see Lemma \ref{lem:ev}(i)). It is well known that in the decomposition of the regular representation into a direct sum of irreducible representations, the number of each irreducible representation is equal to its degree. Thus, by Lemma \ref{lem:vphi} and (\ref{eq:ir-eigen}), we obtain
\be
\label{eq:dec}
\vp = \bigoplus_{W}\left(\sum_{i=1}^s \frac{|K_i|}{d_{\chi_W}}\chi_W(K_i)\right) \mathrm{Id}_W,
\ee
where the direct sum runs over all irreducible representations $(\rho_W, W)$ in the decomposition of $\l_G$ (into a direct sum of irreducible representations), and $\mathrm{Id}_W$ is the identity mapping from $W$ to itself. Therefore, $\dim(\Ker(\vp)) = \sum_{\chi \in \QQ} d_{\chi}^2$. This together with (\ref{eq:ker}) yields $\sum_{\chi \in \QQ} d_{\chi}^2 \ge |S| - 1$ as claimed in (a). 

It remains to prove (b). We may assume $K_{s+1} = \{1\}$ without loss of generality. Choose an arbitrary element $g_j \in K_j$, $1 \le j \le s+1$. (Note that $g_{s+1} = 1$.) By Lemma \ref{lem:partition}(a), each $Cg_{j}$ is a total perfect code in $\Ga$. 

We now prove that the rank of the matrix $(|K_i \cap Cg_j|)_{1 \le i \le m, 1 \le j \le s+1}$ is equal to $s+1$. We show first that 
\be
\label{eq:mij}
|K_i \cap Cg_j| = \left\{ 
\begin{array}{ll}
1, & \mbox{\rm if $1 \le i=j \le s$;} \\ [0.1cm]
0, & \mbox{\rm if $1 \le i \ne j \le s$.}\\
\end{array}
\right.
\ee
Suppose that $K_i \cap Cg_j \ne \emptyset$, where $1 \le i, j \le s$. Then there exist $x \in K_i$ and $y \in C$ such that $x = yg_j$. As $K_i \subseteq S$ and $1 \notin S$, we have $1 \ne x \in S$. Moreover, since $S$ is closed under conjugation and $g_j \in S$, we have $xy^{-1} = yg_j y^{-1} \in S$. Hence $x$ is adjacent to $1$ and $y$ in $\Ga$. Since $1, y \in C$ and $C$ is a total perfect code, it follows that $y = 1$. In particular, $x = g_j \in K_i \cap K_j$. Thus $i = j$ and moreover the only possible common element of $K_i$ and $Cg_i$ is $g_i$. On the other hand, since $1 \in C$, we do have $g_i \in K_i \cap Cg_i$. Therefore, $K_i \cap Cg_i = \{g_i\}$ and \eqref{eq:mij} is proved. 
 
Now suppose $1 \le i \le s+1$ and $j = s+1$, and consider $|K_i \cap Cg_{s+1}| = |K_i \cap C|$ (as $g_{s+1} = 1$). As $K_{s+1} = \{1\}$ and $1 \in C$, we have $|K_{s+1} \cap Cg_{s+1}| = 1$. Since $\Ga(1) = S$ and $1 \in C$ is adjacent to exactly one vertex in $C$, we have $|S \cap C| = 1$. This together with $S = \cup_{i=1}^s K_i$ implies that there exists a unique $i^*$ with $1 \le i^* \le s$ such that $|K_{i^*} \cap C| = 1$ and $|K_{i} \cap C| = 0$ for $1 \le i \le s$ with $i \ne i^*$. Combining this with \eqref{eq:mij}, we obtain that the first $s+1$ rows of $(|K_i \cap Cg_j|)_{1 \le i \le m, 1 \le j \le s+1}$ form the following submatrix:
$$
\bmat{I_{s} & e_{i^*}^T \\ 0 & 1},
$$
where $I_{s}$ is the $s \times s$ identity matrix and $e_{i^*} = (0, \ldots, 0, 1, 0,  \ldots, 0)$ with $1$ in the $i^*$th coordinate. In particular, the rank of $(|K_i \cap Cg_j|)_{1 \le i \le m, 1 \le j \le s+1}$ is equal to $s+1$ as claimed. 

By (\ref{eq:ir-eigen}), $f \in \Ker(\vp_{\pi})$ $\Leftrightarrow$ $f = \sum_{\chi \in \QQ} a_{\chi} \chi$,  for some $a_{\chi} \in \CCC$ $\Leftrightarrow$ $f = \sum_{\chi \in \QQ} a_{\chi} (\sum_{i=1}^m \chi(K_i) \d_{K_i})$ $\Leftrightarrow$ $f = \sum_{i=1}^m (\sum_{\chi \in \QQ} a_{\chi} \chi(K_i)) \d_{K_i}$. Note that $a_{\chi}$ does not rely on $i$. 

Since each $Cg_j$ is a total perfect code in $\Ga$, by Theorem \ref{thm:equ} and the computation above, for $1 \le i \le m$ and $1 \le j \le s+1$, there exist $a_{\chi, j} \in \CCC$ and $\chi \in \QQ$ such that
$$
|K_i \cap Cg_j| = \frac{|K_i|}{d} + \left(\sum_{\chi \in \QQ} a_{\chi, j} \chi(K_i)\right) |K_i|. 
$$
In other words, 
$$
(|K_i \cap Cg_j|)_{1 \le i \le m, 1 \le j \le s+1} = \frac{1}{d} \bmat{|K_1| & \cdots & |K_1|\\ \vdots &  & \vdots \\ |K_m| & \cdots & |K_m|} + (|K_i| \chi(K_i))_{1 \le i \le m, \chi \in \QQ} \cdot (a_{\chi, j})_{\chi \in \QQ, 1 \le j \le s+1}.
$$
Since as shown above the matrix on the left-hand side has rank $s+1$ and the first term on the right-hand side has rank one, the rank of the product $(|K_i| \chi(K_i)) \cdot (a_{\chi, j})$ is at least $s$. Thus the rank of $(|K_i| \chi(K_i))_{1 \le i \le m, \chi \in \QQ}$ is at least $s$. Consequently, $|\QQ| \ge s$ as required in (b). 
\qed
\end{proof}

\begin{example}
\label{ex:circulant}
{\em 
The characters of a cyclic group $C_n = \la a \ra$ of order $n$ are $\chi_{k}$, $0 \le k \le n-1$, defined by $\chi_{k}(a^j) = \omega^{kj}$, where $\omega = e^{2\pi i/n}$. By Corollary \ref{cor:nec-a}, a circulant $\Cay(C_n, S)$ (where $S^{-1} = S \subseteq C_n \setminus \{1\}$) admits a total perfect code only if there are at least $|S|$ integers $k$ between $0$ and $n-1$ such that $\sum_{j: a^{j} \in S} \omega^{kj} = 0$. 
}
\qed
\end{example}

\subsection{Another necessary condition}
\label{subsec:b}

Let $\Ga = \Cay(G, S)$ be the Cayley graph in Theorem \ref{thm:nec-a} and $H$ a subgroup of $G$. Denote by 
$$
\pi(H) = \{x_1 H, \ldots, x_m H\} 
$$
the partition of $G$ into left cosets of $H$ in $G$, where $m = |G:H|$ and $x_1, \ldots, x_m$ are a set of representatives of such cosets. Since these cosets are the $H$-orbits under the action of $H$ on $G$ by right multiplication, and $H$ can be viewed as a subgroup of $\Aut(\Ga)$, it follows that $\pi(H)$ is an equitable partition of $\Ga$. The `standard' basis of $\CCC G_{\pi(H)}$ is $\{\d_{x_1 H}, \ldots, \d_{x_m H}\}$. Let $\l_G^H: G \rightarrow \Sym(\CCC G_{\pi(H)})$ be the representation of $G$ with degree $|G:H|$ defined by $\l_G^H(g)(\d_{x_i H}) = \d_{gx_i H}$ for $g \in G$ and $1 \le i \le m$. The following result is analogous to \cite[Theorem 7]{E87}. 

\begin{theorem}
\label{thm:nec-b}
Let $G$ be a group and $S$ a union of conjugacy classes of $G$ with $1 \not \in S$ and $S^{-1} = S$. Suppose $H$ is a subgroup of $G$ such that $\sum_{K} \frac{|K|\chi(K)}{d_{\chi}} \ne 0$ for every irreducible character $\chi$ of $G$ occurring in the decomposition of $\l_G^H$ (into a direct sum of irreducible characters), where $K$ runs over all conjugacy classes contained in $S$. Then every total perfect code in $\Cay(G, S)$ intersects every left coset of $H$ in $G$ at exactly $|H|/|S|$ elements. In particular, if $|S|$ does not divide $|H|$, then $\Cay(G, S)$ has no total perfect code. 
\end{theorem}

\begin{proof}
As in \textsection \ref{subsec:a}, let $S = \cup_{i=1}^s K_i$ and let $\vp: \CCC G \rightarrow \CCC G$ be the linear mapping defined in (\ref{eq:vphi}) for $V=G$ and $\Ga = \Cay(G, S)$. Similar to Lemma \ref{lem:vphi}, a straightforward computation yields \cite{E87}   
$$
\vp_{\pi(H)} = \sum_{g \in S} \l_{G}^{H}(g).
$$
(That is, the  matrix of $\vp_{\pi(H)}$ is given by $A_{\pi(H)} = \sum_{g \in S} \l_{G}^{H}(g)$ when $\l_{G}^{H}(g)$ is interpreted as the permutation matrix of the permutation $\d_{x_i H} \mapsto \d_{gx_i H}$, $i=1, \ldots, m$ of the basis $\{\d_{x_1 H}, \ldots, \d_{x_m H}\}$ of $\CCC G_{\pi(H)}$.) Based on this and similar to (\ref{eq:dec}), one can verify that
\be
\label{eq:dec-b}
\vp_{\pi(H)} = \bigoplus_{W}\left(\sum_{i=1}^s \frac{|K_i|}{d_{\chi_W}}\chi_W(K_i)\right) \mathrm{Id}_W,
\ee
where the direct sum runs over all irreducible representations $(\rho_W, W)$ in the decomposition of $\l_G^H$ (into a direct sum of irreducible representations). Thus $\det(A_{\pi(H)}) \ne 0$ if and only if $\sum_{i=1}^s \frac{|K_i|}{d_{\chi_W}}\chi_W(K_i) \ne 0$ for every $W$ in (\ref{eq:dec-b}). In this case, by (\ref{eq:eigen})-(\ref{eq:eigen1}) applied to $\pi(H)$, for any total perfect code $C$ in $\Ga$ we have $|x_i H \cap C| = |H|/|S|$, which occurs only when $|S|$ is a divisor of $|H|$. 
\qed
\end{proof}

\medskip

\textbf{Acknowledgements}~~The author is grateful to three anonymous referees for their helpful comments which led to improvement of Theorem \ref{thm:nec-a}(b) and betterment of presentation of the paper. 
This research was supported by the Australian Research Council (FT110100629).

\end{document}